\documentclass{article}
\usepackage{graphicx} % Required for inserting images
\usepackage[utf8]{inputenc}
\usepackage{amsmath}
\usepackage{amssymb}
\usepackage{amsthm}
\usepackage{mathtools}
\usepackage{empheq} 
\usepackage{color}
\usepackage{indentfirst}
\usepackage[dvipdfm,truedimen,left=25truemm,right=25truemm,top=30truemm,bottom=30truemm]{geometry}

\DeclareMathOperator{\mex}{mex}
\DeclareMathOperator{\sg}{sg}
\DeclareMathOperator{\ord}{ord}

\makeatletter
\renewenvironment{proof}[1][\proofname]{\par
  \pushQED{\qed}%
  \normalfont \topsep6\p@\@plus6\p@\relax
  \trivlist
  \item\relax
  {\itshape
  #1\@addpunct{.}}\hspace\labelsep\ignorespaces
}{%
  \popQED\endtrivlist\@endpefalse
}

\theoremstyle{definition}
\newtheorem{lemma}{Lemma}[section]

\newtheorem{theorem}[lemma]{Theorem}

\newtheorem{definition}[lemma]{Definition}
\newtheorem{example}[lemma]{Example}

\title{Sprague-Grundy Value for Common Divisor Nim}
\author{Ryotaro Nohagi\\
nry.kw.no93ann@gmail.com}
\date{}

\begin{document}

\maketitle
\begin{abstract}

We determine the Sprague-Grundy value for the Common Divisor Nim$_m$ (CDN$_m$ for short) for $m \in \mathbb{Z}_{\ge 1}$,  
which is called COMMON DIVISOR game in \cite[Chapter 2, Exercise 8 (p.55)]{Albert}.
\end{abstract}

\section{Introduction}
In this paper, we introduce a finite impartial game, called
Common Divisor
Nim$_{m}$ (CDN$_{m}$ for short) for $m \in \mathbb{Z}_{\ge 1}$. Let $\mathbb{N}$ be the set of all nonnegative integers. 
\begin{definition}\label{def3.1}
    Let $m\in\mathbb{Z}_{\ge 1}$. CDN$_m$ is defined as follows:
    \begin{itemize}
\item[($\mathrm{i}$)] CDN$_m$ is played by two players alternately.  
\item[($\mathrm{ii})$] The set of all positions of CDN$_m$ is $\mathbb{N}^m.$
\item[($\mathrm{iii})$]Let $P=(n_1,\dots,n_m)\in \mathbb{N}^m$ be a position of CDN$_m$. The set $N(P)$ of the next positions to $P$ is given by
  \begin{equation*}
N(P):= \left\{
(n'_{1},\dots,n'_{m}) \in \mathbb{N}^{m} \Biggm|
\begin{array}{lllll} \text{there is an integer } 1\leq i \leq m\\
\text{such that }
n'_j=n_j
 \text{ for }1\leq j\leq m\text{ with }j\neq i\text{, and}\\
n_i-n'_i\text{ is a (positive) common divisor of } n_1,\dots,n_m 
\end{array}\right\};
\end{equation*}
note that $\mathbf{0}:=(0,\dots,0)$ is a unique end position. 
\end{itemize}
\end{definition}

\begin{example}\label{ex3.2}In CDN$_3$, let $P:=(6,3,2)\in\mathbb{N}^3$. Then, $1$ is a unique common divisor of  $6,3$, and $2$. Thus, $N(P)=\{(6,3,1),(6,2,2),(5,3,2)\}$. Next, let $P':=(6,2,2)$. Then, $1$ and $2$ are the common divisors of $6$ and $2$. Thus, $N(P')=\{(5,2,2),(6,1,2), (6,2,1), (4,2,2), (6,0,2), (6,2,0)\}$.
\end{example}

\indent Observe that CDN$_{m}$ is a finite impartial game in the sense of \cite[Chapter 
 I, Definition 4.1 (c)]{siegel}. We know (see, e.g., \cite[Chapter I, Theorem 1.5]{siegel}) that 
 either of two players has a winning
strategy at each position of CDN$_{m}$.
In order to determine which player has a winning strategy at a
position, it suffices to check
if the Sprague-Grundy value of the poistion is equal to $0$ or not (see \cite[Theorem 7.12]{Albert}).
Here the Sprague-Grundy value $\sg(P)$ for $P \in
\mathbb{N}^{m}$ is
recursively defined as follows:
$$ \sg(P) := \mex \{ \sg(P') \mid P' \in N(P) \},$$
where for a finite subset $T$ of $\mathbb{N}$, we define
$\mex (T) := \min (\mathbb{N} \setminus T)$.
The aim of this paper is to determine the Sprague-Grundy value $\sg(P)$ for every position $P\in\mathbb{N}^m$ of  CDN$_{m}$, which refines \cite[Chapter $2$, Exercise $8$ (p.$55$)]{Albert}; 
here, we remark that ``Sprague-Grundy theorem'' (\cite[Chapter I\hspace{-1.2pt}V, Theorem 1.3]{siegel}) is not valid for this game. In \cite{Albert}, it is a problem to find winning strategy, but in this paper, we determine Sprague-Grundy value. Note that to find winning strategy is the corollary of to find Sprague-Grundy value.\\ 
\indent In order to explain our result, we introduce some notation. For $a\in\mathbb{N}$, we define 
\[\ord_2 (a) := 
\begin{cases}
    \max\{L\in\mathbb{N} \mid2^L \text{ is a divisor of } a \} & \text{if }a\neq0 , \\
    +\infty        & \text{if }a=0,
\end{cases}\]
    
\noindent where $+\infty$ is a formal element which is larger than any element in $\mathbb{N}$. For $P=(n_1,\dots,n_m)\in\mathbb{N}^m$, we set  \begin{align*}
\lambda(P) & := \min \bigl\{\ord_{2}(n_{i}) \mid 1 \le i \le m \bigr\}, \\
\iota(P) & := \# \bigl\{ 1 \le i \le m \mid \ord_{2}(n_{i}) = \lambda(P) \bigr\}.
\end{align*} 
\begin{theorem}\label{th1.2}
 For $P=(n_1,\dots,n_m)\in\mathbb{N}^m$, it holds that 
         \[\sg(P) =
\begin{cases}
\lambda(P)+1 & \text{if }P\neq\mathbf{0}\text{ and } \iota(P)\text{ is an odd number},\\
0 & \text{if }P\neq\mathbf{0}\text{ and } \iota(P)\text{ is an even number, or if }P=\mathbf{0}.\\
\end{cases}\] 
\end{theorem}
\section*{\large Acknowledgements}The author would like to thank Professor Daisuke Sagaki, who was his supervisor, for his helpful advice.

\section{Proof of Theorem \ref{th1.2}}
 In what follows, we fix $m\in\mathbb{Z}_{\geq 1}$. Also, keep the notation and seeing in Introduction. 
\begin{lemma}\label{lem3.6}
 Let $P=(n_1,\dots,n_m)\in\mathbb{N}^m$ with  $P\neq\mathbf{0}$, and 
\begin{align*}
    (n_1,\dots, n_{i-1}, n'_i, n_{i+1},\dots,n_m)\in N(P);
\end{align*} 
recall that $n_i-n'_i$ is a common divisor of $n_1,\dots,n_m$. 
Set $l_i:=\ord_2(n_i)\in\mathbb{N}$ and $k_i:=\ord_2(n_i-n'_i)\in\mathbb{N}$; note that $0\leq k_i\leq \lambda(P)\leq l_i$.

\begin{description}
    \item[($\mathrm{i})$] If $l_i> \lambda(P)$, then
\begin{subequations}
\begin{empheq}[left={\ord_2(n'_i) =\empheqlbrace}]{alignat=2}
&0 = \lambda(P) &\qquad \text{if }0=k_i=\lambda(P) ,\label{3.1a} \\
&0 < \lambda(P) &\qquad \text{if }0=k_i<\lambda(P),\label{3.1b}\\
&\lambda(P) &\qquad \text{if }0<k_i=\lambda(P),\label{3.1c}\\
& k_i < \lambda(P)&\qquad \text{if }0<k_i<\lambda(P)\label{3.1d}.
\end{empheq}
\end{subequations}
\item [($\mathrm{ii})$] If $l_i= \lambda(P)$, then
\begin{subequations}
\begin{empheq}[left={\ord_2(n'_i) =\empheqlbrace}]{alignat=2}
&\hspace{5mm}\text{a positive integer or} +\infty\hspace{5mm} > \lambda(P) &\qquad \text{if }0=k_i=\lambda(P),\label{3.2a} \\
&\hspace{21.5mm}0 \hspace{25mm}< \lambda(P) &\qquad \text{if }0=k_i<\lambda(P),\label{3.2b}\\
&\hspace{5mm}\text{a positive integer or} +\infty\hspace{5mm} > \lambda(P)&\qquad \text{if }0<k_i=\lambda(P),\label{3.2c}\\
&\hspace{22mm}k_i \hspace{23mm}< \lambda(P) &\qquad \text{if }0<k_i<\lambda(P)\label{3.2d}.
\end{empheq}
\end{subequations}
\end{description}
\end{lemma}
\begin{proof}
Write $n_i$ as
$n_i=2^{l_i}N_i$, where $N_i$ is an odd number. Also, write $n_i-n'_i$ as $n_i-n'_i=2^{k_i}N'_i$, where $N'_i$ is an odd number. Note that $n'_i=n_i-(n_i-n'_i)=2^{l_i}N_i-2^{k_i}N'_i=2^{k_i}(2^{l_i-k_i}N_i-N'_i)$.  
\begin{description}
    \item[(a)] If $0=k_i=\lambda(P)$, then
\begin{align*}
&\ord_2(n'_i)=\ord_2(2^0(2^{l_i}N_i-N'_i))\\
&=\begin{cases}
\ord_2(2^{l_i}N_i-N'_i)=0=\lambda(P) &\text{if }\lambda(P)< l_i;\text{ cf.\,\eqref{3.1a}},\\
\ord_2(N_i-N'_i)>0=\lambda(P) &\text{if }\lambda(P)=l_i;\text{ cf.\,\eqref{3.2a}}.
\end{cases}
\end{align*}
    \item[(b)] If $0=k_i<\lambda(P)$, then
\begin{align*}
&\ord_2(n'_i)=\ord_2(2^0(2^{l_i}N_i-N'_i))\\
&=\begin{cases}
\ord_2(2^{l_i}N_i-N'_i)=0<\lambda(P) &\text{if }\lambda(P)< l_i;\text{ cf.\,\eqref{3.1b}},\\
\ord_2(2^{\lambda(P)}N_i-N'_i)=0<\lambda(P) &\text{if }\lambda(P)=l_i;\text{ cf.\,\eqref{3.2b}}.
\end{cases}
\end{align*}
    \item[(c)] If $0<k_i=\lambda(P)$, then
\begin{align*}
&\ord_2(n'_i)=\ord_2(2^{\lambda(P)}(2^{l_i-\lambda(P)}N_i-N'_i))\\
&=\begin{cases}
\lambda(P) &\text{if }\lambda(P)< l_i;\text{ cf.\,\eqref{3.1c}},\\
\ord_2(2^{\lambda(P)}(2^{\lambda(P)-\lambda(P)}N_i-N'_i))\geq l{_i}+1>\lambda(P) &\text{if }\lambda(P)=l_i;\text{ cf.\,\eqref{3.2c}}.
\end{cases}
\end{align*}
    \item[(d)] If $0<k_i<\lambda(P)$, then
\begin{align*}
&\ord_2(n'_i)
=\ord_2(2^{k_i}(2^{l_i-k_i}N_i-N'_i))\\
&=\begin{cases}
k_i<\lambda(P) &\text{if }\lambda(P)< l_i;\text{ cf.\,\eqref{3.1d}},\\
\ord_2(2^{k_i}(2^{\lambda(P)-k_i}N_i-N'_i))=k_i<\lambda(P) &\text{if }\lambda(P)=l_i;\text{ cf.\,\eqref{3.2d}}.
\end{cases}
\end{align*}
\end{description}
Thus we have proved the lemma.
\end{proof}
Now, we prove Theorem \ref{th1.2}. For $P\in\mathbb{N}^m$, we set
     \[\varphi(P) :=
\begin{cases}
\lambda(P)+1 & \text
{if }P\neq\mathbf{0}\text{ and } \iota(P)\text{ is an odd number},\\
0 & \text{if }P\neq\mathbf{0}\text{ and } \iota(P)\text{ is an even number, or if } P=\mathbf{0}.\\
\end{cases}\]
We show that $\varphi(P)=\sg(P)$ for all $P\in\mathbb{N}^m$. If $P=\mathbf{0}$, then assertion is obvious. Assume that $P\neq\mathbf{0}$. By the definition of $\sg(P)$, it suffices to show the following claim (a) and (b):
\begin{itemize}
    \item[(a)] For any $P'\in N(P)$, it holds that $\varphi(P')\neq\varphi(P)$.
    \item[(b)] For any  integer $\alpha$ with  $0\leq \alpha<\varphi(P)$, there exists some $P'\in N(P)$ such that $\alpha=\varphi(P')$. 
\end{itemize}
First, we show  (a). Let  $P=(n_1,\dots,n_m)\in\mathbb{N}^m$ with  $P\neq\mathbf{0}$, and $P'=(n_1,\dots,n'_i,\dots,n_m)\in N(P)$ with $n'_i<n_i$.
\paragraph{Case A1.} Assume that $\iota(P)$ is an odd number; note that $\varphi(P)>0$.  If $\lambda(P)<\ord_2(n_i)$, then it follows from Lemma \ref{lem3.6} that \begin{subequations}
\begin{empheq}[left={\ord_2(n'_i) =\empheqlbrace}]{alignat=2}
&0 = \lambda(P) &\qquad \text{if }0=\ord_2(n_i-n'_i)=\lambda(P),\label{3.3a}\\
&0 < \lambda(P) &\qquad \text{if }0=\ord_2(n_i-n'_i)<\lambda(P),\label{3.3b}\\
&\lambda(P) &\qquad \text{if }0<\ord_2(n_i-n'_i)=\lambda(P),\label{3.3c}\\
&k_i < \lambda(P) &\qquad \text{if }0<\ord_2(n_i-n'_i)<\lambda(P).\label{3.3d}
\end{empheq}    
\end{subequations} 
If $\lambda(P)=\ord_2(n_{i}-n'_i)$, then $\ord_2(n_i)>\lambda(P)=\ord_2(n'_i)=\lambda(P')$ by \eqref{3.3a} and \eqref{3.3c}. Therefore, we get $\iota(P')=\iota(P)+1\in2\mathbb{Z}$. Thus, we have $\varphi(P')=0<1\leq\lambda(P)+1=\varphi(P)$. If $\ord_2(n_{i}-n'_i)<\lambda(P)$, then $\lambda(P)>\ord_2(n'_i)=\lambda(P')$ by \eqref{3.3b} and \eqref{3.3d}. Therefore we get $\iota(P')=1\in2\mathbb{Z}+1$. Thus, we have $\varphi(P')=\lambda(P')+1<\lambda(P)+1=\varphi(P)$.\\
\indent If $ \lambda(P)=\ord_2(n_i)$, then it follows from lemma \ref{lem3.6} that
\begin{subequations}
\begin{empheq}[left={\ord_2(n'_i) =\empheqlbrace}]{alignat=2}
&\hspace{5mm}\text{a positive integer or} +\infty\hspace{5mm} > \lambda(P) &\qquad \text{if }0=\ord_2(n_i-n'_i)=\lambda(P),\label{3.4a} \\
&\hspace{22mm}0 \hspace{25mm}< \lambda(P) &\qquad \text{if }0=\ord_2(n_i-n'_i)<\lambda(P),\label{3.4b}\\
&\hspace{5mm}\text{a positive integer or} +\infty\hspace{5mm} > \lambda(P)&\qquad \text{if }0<\ord_2(n_i-n'_i)=\lambda(P),\label{3.4c}\\
&\hspace{22mm}k_i \hspace{23mm} <\lambda(P) &\qquad \text{if }0<\ord_2(n_i-n'_i)<\lambda(P)\label{3.4d}.
\end{empheq}
\end{subequations}
If $\ord_2(n_i-n'_i)=\lambda(P)$, then
\begin{align}
     \ord_2(n'_i)>\lambda(P)=\ord_2(n_i)=\ord_2(n_i-n'_i)\label{3.7}
 \end{align} by \eqref{3.4a} and \eqref{3.4c}. If $P'=\mathbf{0}$ or $\iota(P')\in2\mathbb{Z}$, then $\varphi(P')=0<1\leq\varphi(P)$. If $\iota(P)\geq3$, then $\lambda(P)=\lambda(P')$ and  $\iota(P')=\iota(P)-1\in2\mathbb{Z}$. Thus, we obtain $\varphi(P')=0<1\leq\varphi(P)$. If $P'\neq\mathbf{0}$, $\iota(P')\in2\mathbb{Z}+1$, and $\iota(P)=1$, then $\lambda(P)<\lambda(P')$ by \eqref{3.7}. Thus, we have $\varphi(P')=\lambda(P')+1>\lambda(P)+1=\varphi(P)$.\\
 \indent If $\ord_2(n_i-n'_i)<\lambda(P)$, then $\lambda(P)>\ord_2(n'_i)=\lambda(P)$. Thus,  $\iota(P)=1$ and $\varphi(P')=\lambda(P')+1<\lambda(P)+1=\varphi(P)$.
 \paragraph{Case A2.}Assume that $\iota(P)$ is an even number. If $\lambda(P)<\ord_2(n_i)$, then we see by the same argument as above that $\ord_2(n'_i)$ is equal to one of \eqref{3.3a} -- \eqref{3.3d}. If $\lambda(P)=\ord_2(n_{i}-n'_i)$, then $\ord_2(n_i)>\lambda(P)=\ord_2(n'_i)=\lambda(P')$ by \eqref{3.3a} and \eqref{3.3c}. Thus, $\iota(P')=\iota(P)+1\in2\mathbb{Z}+1$ holds. Therefore, we obtain $\varphi(P')=\lambda(P')+1\geq1>0=\varphi(P)$. Otherwise, If $\ord_2(n_{i}-n'_i)<\lambda(P)$, then $\lambda(P)>\ord_2(n'_i)=\lambda(P')$ by \eqref{3.3b} and \eqref{3.3d}. Thus, $\lambda(P)>\ord_2(n'_i)=\lambda(P')$ holds. Therefore, we have $\iota(P')=1\in2\mathbb{Z}+1$ and $\varphi(P')=\lambda(P')+1\geq1>0=\varphi(P)$.\\
 \indent If $\lambda(P)=\ord_2(n_i)$, then we see by the same argument as above that $\ord_2(n'_i)$ is equal to one of \eqref{3.4a} -- \eqref{3.4d}. If $\ord_2(n_{i}-n'_i)=\lambda(P)$, then $\ord_2(n'_i)>\lambda(P)=\ord_2(n_i)=\ord_2(n_i-n'_i)$ by \eqref{3.4a} and \eqref{3.4c}. Since  $\lambda(P')=\lambda(P)$ by $\iota(P)\geq2$, it follows that $\iota(P')=\iota(P)-1\in2\mathbb{Z}+1$. Therefore, we get $\varphi(P')=\lambda(P')+1=\lambda(P)+1>0=\varphi(P)$. If $\ord_2(n_{i}-n'_i)<\lambda(P)$, then $\lambda(P)>\ord_2(n'_i)=\lambda(P')$ by \eqref{3.4b} and \eqref{3.4d}. Thus, we have $\iota(P)=1\in2\mathbb{Z}+1$. Therefore, we get $\varphi(P')=\lambda(P')+1>\lambda(P)+1=\varphi(P)$.\\
 \indent Thus we have shown (a).\\
 \indent Next, we show  (b). Let $P=(n_1,\dots,n_m)\in\mathbb{N}^m$. If $\varphi(P)=0$, then there is nothing to prove. Assume that  $\varphi(P)>0$, and let $0\leq\alpha<\varphi(P)$. Since $\varphi(P)>0$, it follows from the definition of $\varphi(P)$ that ($P\neq\mathbf{0}$, $\iota(P)\in2\mathbb{Z}+1$, and)  $\varphi(P)=\lambda(P)+1$. Hence we get $0\leq\alpha\leq \lambda(P)$. For $1\leq j\leq m$, if $n_{j}\neq0$, then we define $l_{j}\in\mathbb{N}$ and $N_{j}\in 2\mathbb{Z}+1$ by $n_{j}=2^{l_j}N_j$;  if $n_j=0$, we set $l_j=+\infty$, $N_j=0$. 
 \paragraph{Case B1.} Assume that $0<\alpha\leq \lambda(P)$. If we set $h:=2^{\alpha-1}$, then $h$ divides $n_j$ for every $1\leq j\leq m$. Let $1\leq s\leq m$ be such that $l_{s}=\ord_2(n_s)$ is equal to $\lambda(P)$. We have 
 \begin{align*}
    P':=(n_1,\dots, \underbrace{n_{s}-h}_{=:n'_s},\dots,n_m)\in N(P).
\end{align*}
We show that $\varphi(P')=\alpha$. Indeed, since $\lambda(P)-\alpha+1>0$, it follows that 
\begin{align*}
    \ord_2(n'_s)=\ord_2(2^{\lambda(P)}N_{s}-2^{\alpha-1})=\ord_2(2^{\alpha-1}(2^{\lambda(P)-\alpha+1}N_{s}-1))=\alpha-1<\lambda(P).
\end{align*}
Hence, we see that $\lambda(P')=\alpha-1<\lambda(P)$ and $\iota(P')=1$. Therefore we obtain  $\varphi(P')=\lambda(P')+1=\alpha$, as desired.
\paragraph{Case B2.} Assume that $\alpha = 0$. If $\# \{ 1 \le i \le m \mid n_{i} > 0\} = 1$, then $P':=\mathbf{0} \in N(P)$ satisfies $\varphi(P')=\varphi(\mathbf{0}) = 0 = \alpha$. Assume that $\# \{ 1 \le i \le m \mid n_{i} > 0\} \ge 2$. Since $\varphi(P) > 0$, it follows that $\iota(P)$ is an odd number. It suffices to show that $\iota(P')\in2\mathbb{Z}$ for some $P'\in N(P)$. Let $1\leq s\leq m$ be such that $l_s=\ord_2(n_s)$ is equal to $\lambda(P)$.
 If $\iota(P) \ge 3$, then we consider;
\begin{equation*}
P':=(n_{1},\dots, \underbrace{ n_{s}-2^{\lambda(P)} }_{=:n_{s}'},
\dots,n_{m}) \in N(P).
\end{equation*}
We see that
\begin{align*}
    \ord_{2}(n_{s}')=\ord_2(2^{\lambda(P)}N_{s}-2^{\lambda(P)})=\ord_2(2^{\lambda(P)}(N_{s}-1)) > \lambda(P) = \ord_{2}(n_{s}).
\end{align*}
Since $\iota(P)\geq3$,
we get $\lambda(P') = \lambda(P)$ and $\iota(P') = \iota(P) - 1 \in 2 \mathbb{Z}$, as desired.\\
\indent Assume that $\iota(P) = 1$. Let $1\leq s\leq m$ be as above, and set
\begin{align*}
&\mu(P):=\min \{ \ord_{2}(n_{i}) \mid 1 \le i \le m,  i \ne s \} > \lambda(P),\\
&\kappa(P):= \# \{ 1 \le i \le m \mid \ord_{2}(n_{i}) = \mu(P) \} \ge 1;
\end{align*} 
 note that $\mu(P) < + \infty$, because $\# \{ 1 \le i \le m \mid n_{i} > 0\} \ge 2$. Let $1 \le t \le m$ be such that $l_t=\ord_{2}(n_{t})$ is equal to $\mu(P)$. Consider 
 \begin{equation*}
P'=(n_{1},\dots, \underbrace{ n_{t}-2^{\lambda(P)} }_{=:n_{t}'},
\dots,n_{m}) \in N(P).
\end{equation*}
We see that 
\begin{align*}
    \ord_{2}(n_{t}') =\ord_2(2^{\mu(P)}N_{t}-2^{\lambda(P)})=\ord_2(2^{\lambda(P)}(2^{\mu(P)-\lambda(P)}M-1))= \lambda(P).
\end{align*} 
Therefore, we obtain $\lambda(P') = \lambda(P)$ and $\iota(P') = \iota(P) + 1 =2$, as desired.\\
\indent Thus, we have shown (b), thereby completing the proof of Theorem \ref{th1.2}.

\end{document}